\documentclass[11pt]{article}

\usepackage{amssymb,amsmath,amsfonts,amsthm}
\usepackage{latexsym}
\usepackage{graphics}
\usepackage{indentfirst}
\usepackage{hyperref}

\newtheorem{innercustomthm}{Theorem}
\newenvironment{customthm}[1]
  {\renewcommand\theinnercustomthm{#1}\innercustomthm}
  {\endinnercustomthm}
\newtheorem{innercustomlem}{Lemma}
\newenvironment{customlem}[1]
  {\renewcommand\theinnercustomlem{#1}\innercustomlem}
  {\endinnercustomlem}

\newtheorem*{thm*}{Theorem}
\newtheorem{thm}{Theorem}
\newtheorem{lem}[thm]{Lemma}
\newtheorem{pro}[thm]{Proposition}

\newtheorem{cor}[thm]{Corollary}

\newtheorem*{PLCconj}{Partial List Coloring Conjecture}

\newtheorem{ques}[thm]{Question}

\newcommand{\N}{\mathbb{N}}

\newcommand{\col}{\mathrm{col}}

\begin{document}

\title{Partial DP-Coloring of Graphs}

\author{Hemanshu Kaul\footnote{Department of Applied Mathematics, Illinois Institute of Technology, Chicago, IL 60616. E-mail: {\tt kaul@iit.edu}} \\
Jeffrey A. Mudrock\footnote{Department of Mathematics, College of Lake County, Grayslake, IL 60030. E-mail: {\tt jmudrock@clcillinois.edu}} \\
Michael J. Pelsmajer\footnote{Department of Applied Mathematics, Illinois Institute of Technology, Chicago, IL 60616. E-mail: {\tt pelsmajer@iit.edu}} }

\maketitle

\begin{abstract}

\medskip

In 1980, Albertson and Berman introduced \emph{partial coloring}.  In 2000, Albertson, Grossman, and Haas introduced \emph{partial list coloring}.  Here, we initiate the study of partial coloring for an insightful generalization of list coloring introduced in 2015 by Dvo\v{r}\'{a}k and Postle, \emph{DP-coloring} (or \emph{correspondence coloring}).  We consider the DP-coloring analogue of the Partial List Coloring Conjecture, which generalizes a natural bound for partial coloring.  We show that while this partial DP-coloring conjecture does not hold, several results on partial list coloring can be extended to partial DP-coloring. We also study partial DP-coloring of the join of a graph with a complete graph, and we present several interesting open questions.



\end{abstract}

\section{Introduction}\label{intro}

In this paper all graphs are nonempty, finite, simple graphs unless otherwise noted.  Generally speaking we follow West~\cite{W01} for terminology and notation.  The set of natural numbers is $\N = \{1,2,3, \ldots \}$.  For $m \in \N$, we write $[m]$ for the set $\{1, \ldots, m \}$.  If an edge in $E(G)$ connects the vertices $u$ and $v$, the edge can be represented by $uv$ or $vu$.  If $G$ is a graph and $S, U \subseteq V(G)$, we use $G[S]$ for the subgraph of $G$ induced by $S$, and we use $E_G(S, U)$ for the subset of $E(G)$ with at least one endpoint in $S$ and at least one endpoint in $U$.  We use $\alpha(G)$ and $\omega(G)$ for the size of the largest independent set and the size of the largest clique in $G$ respectively.  For $v \in V(G)$, we write $d_G(v)$ for the degree of vertex $v$ in the graph $G$, and we use $\Delta(G)$ and $\delta(G)$ for the maximum and minimum degree of a vertex in $G$ respectively.  For vertex disjoint graphs $G_1$ and $G_2$, we write $G_1 \vee G_2$ for their join. 
Given a set $A$, $\mathcal{P}(A)$ is the power set of $A$.

\subsection{Partial List Coloring}

In the classical vertex coloring problem, colors must be assigned to the vertices of a graph $G$ so that adjacent vertices receive different colors.  The assignment is a \emph{proper $m$-coloring} if colors come from an $m$-set such as $[m]$; the smallest such $m$ is the \emph{chromatic number}, denoted $\chi(G)$.
Given fewer than $\chi(G)$ colors, we might instead try to color as many vertices as possible: this is \emph{partial coloring}, introduced by Albertson and Berman~\cite{AB80} in~1980.  The \emph{partial $t$-chromatic number} of a graph $G$, denoted $\alpha_t(G)$, is the maximum number of vertices that can be colored with $t$ colors.  

List coloring is a well known variation of graph coloring, introduced independently by Vizing~\cite{V76} and Erd\H{o}s, Rubin, and Taylor~\cite{ET79} in the 1970s.
Each vertex $v$ has a list $L(v)$ of allowable colors; $L$ is called a \emph{list assignment} (or \emph{$m$-assignment} if every list has size $m$), and an \emph{$L$-coloring} (or \emph{proper $L$-coloring}) is a proper coloring where each vertex $v$ is assigned a color from its list $L(v)$.  A graph $G$ is \emph{$m$-choosable} if there is an $L$-coloring whenever $L$ is an $m$-assignment; the minimum such $m$ for a graph $G$ is the \emph{list chromatic number} of a graph $G$, denoted $\chi_\ell(G)$.

Albertson, Grossman, and Haas~\cite{AGH00} introduced \emph{partial list coloring} with a ``frankly mischievous'' intent of inciting further work.
Indeed, this has received attention in several papers~\cite{AGH00, C99, HH03, I10, J01, JM15, V93}.
Given a list assignment $L$, we want to properly $L$-color as many vertices as possible.  Let $\alpha_L(G)$ be the maximum size of a subset of vertices of $G$ that can be properly $L$-colored.\footnote{To be clear, we say that a subset of vertices $V'$ is properly $L$-colored if each vertex $v\in V'$ is assigned a color from its list $L(v)$ such that adjacent vertices receive different colors.  Thus, $\alpha_L(G)$ is the maximum order of an induced subgraph $G[S]$ that has an $L'$-coloring, where $L'$ is the list assignment $L$ restricted to $S$.}  The \emph{partial $t$-choice number} of a graph $G$, denoted $\alpha_t^\ell(G)$, is the minimum of $\alpha_L(G)$ over all $t$-assignments $L$ for $G$.\footnote{Elsewhere in the literature, the partial $t$-choice number of $G$ is denoted $\lambda_t(G)$.}

Note that $\alpha_1(G)=\alpha(G)$, $\alpha_t(G)=|V(G)|$ for $t\ge \chi(G)$, and $\alpha_t(G) \geq t|V(G)|/\chi(G)$ for $t \in [\chi(G)]$ by taking the largest $t$ color classes from a proper $\chi(G)$-coloring.
A question that has received considerable attention is whether the simple bound $\alpha_t(G) \geq t|V(G)|/\chi(G)$ can be extended to partial list coloring.

\begin{PLCconj} [\cite{AGH00}] \label{conj: PLCC}
For any graph $G$ and $t \in [\chi_\ell(G)]$,
$$\alpha_t^\ell(G) \geq \frac{t|V(G)|}{\chi_\ell(G)}.$$
\end{PLCconj}

Since $\alpha_t^\ell(G)=|V(G)|$ for $t\ge \chi_\ell(G)$, the conjecture is true for $t=\chi_\ell(G)$.
One can easily verify that $\alpha_1^\ell(G) = \alpha(G)$; hence it holds for $t=1$ as well.
While the Partial List Coloring Conjecture is still open, some weaker general lower bounds are known: for all graphs $G$ and $t\in [\chi_\ell(G)]$, we have
$\alpha_t^\ell(G) \geq |V(G)|\left(1-\left(1- \frac{1}{\chi(G)} \right)^t \right)$ (see~\cite{AGH00,V93}), $\alpha_t^\ell(G) > {\frac{6}{7}}(t|V(G)|/\chi_\ell(G))$ (see~\cite{C99}), and $\alpha_t^\ell(G) \geq |V(G)|/(\lceil \chi_\ell(G)/t \rceil)$ (see~\cite{HH03, I10}).  The Partial List Coloring Conjecture has been proven for bipartite graphs (by the bound from~\cite{AGH00,V93}), graphs $G$ with $\Delta(G) \leq \chi_\ell(G)$ (see~\cite{J01}), claw-free graphs, chordless graphs, chordal graphs, series parallel graphs, and graphs $G$ satisfying $|V(G)| \leq 2 \chi(G) + 1$ (see~\cite{JM15}).  Iradmusa~\cite{I10} also showed that for every graph $G$, $\alpha_t^\ell(G) \geq \frac{t|V(G)|}{\chi_\ell(G)}$ holds for at least half the values of $t$ in $[\chi_\ell(G)-1]$.

Note that if every graph $G$ contained a $t$-choosable induced subgraph of order at least $t|V(G)|/\chi_\ell(G)$ whenever $t \in [\chi_\ell(G)]$, it would immediately imply the Partial List Coloring Conjecture.  However, that statement does not hold true, since it is known that there is an infinite family of 3-choosable graphs $G$ for which the largest induced 2-choosable subgraph has order at most $5|V(G)|/8$ (see~\cite{JM15}).  For such graphs, if Conjecture~\ref{conj: PLCC} is indeed true, different $t$-assignments $L$ will sometimes require different vertex subsets to be colored in order to achieve the required bound.

\subsection{Partial DP-Coloring}

Our goal is to extend the spirit of mischief in~\cite{AGH00} to DP-coloring.
\emph{DP-coloring} is an insightful generalization of list coloring, introduced in 2015 by Dvo\v{r}\'{a}k and Postle~\cite{DP15}. They created DP-coloring, which they called \emph{correspondence coloring}, as part of a proof that planar graphs without cycles of lengths 4 to 8 are 3-choosable. DP-coloring has been extensively studied over the past 5 years (see e.g.,~\cite{B16,B17,BK182,BK19,BK17,BK18, BK172, KM19, KM20, KO18, KO182, LL19, LLYY19, M18}). 
Intuitively, DP-coloring considers the worst-case scenario of how many colors we need in the lists if we no longer can identify the names of the colors.  Formally, DP-colorings are defined in terms of $m$-fold covers, which are now stated, following~\cite{BK17}:  A \emph{cover} of a graph $G$ is a pair $\mathcal{H} = (L,H)$ consisting of a graph $H$ and a function $L: V(G) \rightarrow \mathcal{P}(V(H))$ satisfying the following four requirements:

\vspace{2mm}

\noindent(1) the sets $\{L(u) : u \in V(G) \}$ form a partition of $V(H)$, \\
(2) for every $u \in V(G)$, the graph $H[L(u)]$ is complete, \\
(3) for every $uv \in E(G)$, the edge set $E_H(L(u),L(v))$ is a matching (possibly empty), \\
(4) for $u,v\in V(G)$ with $uv\not\in E(G)$ and $u\not=v$, $E_H(L(u),L(v))=\emptyset$.

\vspace{2mm}

\noindent
Furthermore, $\mathcal{H}$ is an \emph{$m$-fold cover} if $|L(u)|=m$ for each $u \in V(G)$.
An \emph{$\mathcal{H}$-coloring} of $G$ is defined to be an independent set in $H$ of size $|V(G)|$, i.e., an independent set in $H$ with exactly one vertex in $L(u)$ for each $u\in V(G)$.
The \emph{DP-chromatic number} of a graph $G$, denoted $\chi_{DP}(G)$, is the smallest $m$ such that every $m$-fold cover $\mathcal{H}$ of $G$ has an $\mathcal{H}$-coloring.

DP-coloring generalizes the following model of list coloring (as explained in~\cite{BK17}): Given a graph $G$ and list assignment $L$, create a new graph $H$ such that each list $L(v)$ is a clique in $H$ and for each edge $uv$ in $G$, there is a matching from $L(u)$ to $L(v)$ that joins pairs of vertices with the same color.  Thus, independent sets in $H$ of size $|V(G)|$ correspond to proper $L$-colorings of $G$.  In particular, each edge between cliques $L(u)$ and $L(v)$ in $H$ represents one color that must be forbidden from being chosen for both $u$ and $v$. DP-coloring generalizes that model by allowing \emph{arbitrary} matchings between each pair of cliques $L(u)$ and $L(v)$.  

Since $L$-colorings for any $m$-assignment $L$ for $G$ can be modeled by $\mathcal{H}$-colorings of an $m$-fold cover $\mathcal{H}$, as described previously, $\chi_\ell(G) \leq \chi_{DP}(G)$.  This inequality can be strict; for example, $\chi_{DP}(C_n) = 3$ for any cycle $C_n$~\cite{BK17} but $\chi_\ell(C_n)=2$ when $n$ is even~\cite{ET79}.  It is also easy to show that $\chi_{DP}(G) \leq \col(G)$ where $\col(G)$ is the usual \emph{coloring number}\footnote{The coloring number of a graph $G$ is the smallest integer $d$ such that there exists an ordering $v_1, \ldots, v_n$ of the vertices of $G$ so that $v_i$ has at most $d-1$ neighbors preceeding it in the ordering for each $i \in [n]$.} of the graph $G$.  Dvo\v{r}\'{a}k and Postle~\cite{DP15} observed that Brooks' Theorem extends to DP-coloring: $\chi_{DP}(G) \leq \Delta(G)$ provided that $G$ is connected and neither a cycle nor a complete graph.

We now define partial DP-coloring so that it generalizes partial list coloring in the natural way.
Given a cover $\mathcal{H} = (L,H)$, in particular a $t$-fold-cover with $t<\chi_{DP}(G)$, we wish to $\mathcal{H}$-color as many vertices as possible; i.e., find an independent set in $H$ of size as large as possible.  Thus, we define the \emph{partial DP $t$-chromatic number} of a graph $G$, denoted $\alpha_t^{DP}(G)$, to be the minimum of $\alpha(H)$ taken over all $t$-fold covers $\mathcal{H} = (L,H)$ of $G$.

We now make some basic observations.  For DP-chromatic number and for partial DP $t$-chromatic number, it suffices to consider coverings where $E_H(L(u),L(v))$ is a perfect matching for all $uv\in E(G)$.  Note that $\alpha_t^{DP}(G)=|V(G)|$ if and only if $t \ge \chi_{DP}(G)$, $\alpha_1^{DP}(G)=\alpha(G)$, and $\alpha_t^{DP}(G) \leq \alpha_t^\ell(G) \leq \alpha_t(G)$.  For any induced subgraph $G[S]$ with DP-chromatic number at most $t$, we have $\alpha_t^{DP}(G) \ge |S|$; this leads us to ask the following question.

\begin{ques} \label{ques: induced}
For any graph $G$ and $t \in \N$, does $\alpha_t^{DP}(G)$ always equal the largest possible order of an induced subgraph of $G$ with DP-chromatic number at most $t$?
\end{ques}

We will see that the answer to Question~\ref{ques: induced} is ``no''.
Another natural question is whether the bound $\alpha_t(G)\ge t|V(G)|/\chi(G)$ has a DP-coloring analogue, similar to the Partial List Coloring Conjecture.

\begin{ques} \label{ques: PDPCC}
For any graph $G$, is it always the case that
\begin{equation}
\alpha_t^{DP}(G) \geq \frac{t|V(G)|}{\chi_{DP}(G)} \tag{$*$}
\end{equation}
for all $t \in [\chi_{DP}(G)]$?
\end{ques}

Since $\alpha_t^{DP}(G)=|V(G)|$ for $t\ge \chi_{DP}(G)$, ($*$) holds for $t=\chi_{DP}(G)$.
Since $\alpha_1^{DP}(G)=\alpha(G)$, $\chi_{DP}(G)\ge\chi(G)$, and $\alpha(G)\ge |V(G)|/\chi(G)$, ($*$) holds for $t=1$.
Thus, the answer to Question~\ref{ques: PDPCC} is yes for any graph $G$ with $\chi_{DP}(G)\le 2$,
and for a graph $G$ with $\chi_{DP}(G)=3$, the answer is yes if and only if $\alpha_2^{DP}(G) \geq \frac{2|V(G)|}{\chi_{DP}(G)}$.
However, in general the answer to Question~\ref{ques: PDPCC} is no: there are graphs $G$ for which ($*$) does not hold.
Thusly motivated, we will say that a graph is \emph{partially DP-nice} if~($*$) holds true for all $t \in [\chi_{DP}(G)]$.
We further explore these concepts in the rest of this paper.

\subsection{Outline of the Paper and Open Questions}

In Section~\ref{twofold} we study $2$-fold covers.  We will show that $\alpha_2^{DP}(V_8)=6$ for the \emph{Wagner graph} $V_8$
and that $V_8$ has no induced subgraph with DP-chromatic number 2 and order at least $2|V(V_8)|/3$, which answers Question~\ref{ques: induced}.  Additionally, this will show that $V_8$ is partially DP-nice.  We answer Question~\ref{ques: PDPCC} by presenting several examples of graphs $G$ with $\chi_{DP}(G)=3$ and $\alpha_2^{DP}(G) < 2|V(G)|/3$, including an infinite family of graphs on $5n$ vertices, and the cube graph $Q_3$.  

We also consider planar graphs and observe that any nontrivial planar graph $G$ of girth at least~5 is partially DP-nice.  

Using the aforementioned results involving $V_8$ and $Q_3$, along with results on the feedback vertex number, we are able to characterize $Q_3$ as the only connected, subcubic, triangle -free graph that is not partially DP-nice.

\begin{thm} \label{thm: subcubic}
Suppose $G$ is a connected, subcubic, triangle-free graph, and $G \neq Q_3$.  Then, $G$ is partially DP-nice.
\end{thm}  

Since all our examples that violate~($*$) have $t=2$ and $\chi_{DP}(G)=3$, the following questions are natural.

\begin{ques} \label{ques: twofold}
Are there graphs $G$ such that $\alpha_2^{DP}(G) < 2|V(G)|/\chi_{DP}(G)$ with $\chi_{DP}(G) > 3$?
\end{ques}

\begin{ques} \label{ques: big}
For each $t \geq 4$, does there exist a graph $G$ such that $\chi_{DP}(G) = t$ and $G$ is not partially DP-nice?
\end{ques}

In Section~\ref{subadd}, we show that many of the ideas in~\cite{I10} generalize from list coloring to DP-coloring. The main tool in Section~\ref{subadd} is a subadditivity lemma.

\begin{lem} \label{lem: triangle}
	For any graph $G$ and $t_1, \ldots, t_k \in \N$, $$\alpha_{t}^{DP}(G) \leq \sum_{i=1}^k \alpha_{t_i}^{DP}(G),$$
	where $t=\sum_{i=1}^k t_i$.
\end{lem}

This allows us to prove Theorem~\ref{cor: ceiling}, a bound for all $t$ that is in some sense close to the inequality~$(*)$, and Theorem~\ref{cor: half the values}, which shows that the inequality~$(*)$ holds for many values of~$t$.

\begin{thm} \label{cor: ceiling}
For any graph $G$ and $t \in [\chi_{DP}(G)]$,
$$\alpha_t^{DP}(G) \geq \frac{|V(G)|}{\lceil \chi_{DP}(G)/t \rceil}.$$
\end{thm}

It follows that~($*$) holds true whenever $t$ divides $\chi_{DP}(G)$.
Hence, Question~\ref{ques: twofold} can be restricted to graphs $G$ where $\chi_{DP}(G)$ is odd.

\begin{thm} \label{cor: half the values}
For any graph $G$, the inequality $\alpha_t^{DP}(G) \geq t|V(G)|/\chi_{DP}(G)$ holds true for at least half of the values of $t$ in $[\chi_{DP}(G)-1]$.
\end{thm}

In Section~\ref{nice} we prove various classes of graphs are partially DP-nice, including chordal graphs and series-parallel graphs.  We also consider the join of a graph with a complete graph. Specifically, we use Bernshteyn, Kostochka, and Zhu's recent result~\cite{BK18} that for any graph $G$ there exists a threshold $N \leq 3|E(G)|$ such that $\chi_{DP}(G \vee K_p) = \chi(G \vee K_p)$ whenever $p \geq N$ to motivate and also help answer the question as to whether for such graphs, partial coloring and partial DP-coloring are similarly related.

\begin{thm} \label{thm: join}
For any graph $G$, there exists $p \in \N$ such that $G \vee K_p$ is partially DP-nice.
\end{thm}

\section{Two-Fold Covers} \label{twofold}

A \emph{feedback vertex set} is a set of vertices in a graph whose removal yields an acyclic graph.  The minimum size of a feedback vertex set in a graph $G$ is called the \emph{feedback vertex number} and is denoted $\tau(G)$.  An acyclic graph has coloring number at most~2, so it also has DP-chromatic number at most~2.  Since every graph $G$ has an induced acyclic subgraph of order $|V(G)|-\tau(G)$ with DP-chromatic number at most~2, we get the general bound
$\alpha_2^{DP}(G) \ge |V(G)|-\tau(G)$.

Combining this observation with known upper bounds on $\tau(G)$ yields immediate results.  If $G$ is a planar graph then $\tau(G)\leq 3|V(G)|/5$ by~\cite{B76}, so $\alpha_2^{DP}(G) \ge 2|V(G)|/5$.  There exist planar graphs with DP-chromatic number~5 (see~\cite{DP15}), so we see that such graphs satisfy inequality~($*$) for $t=2$.

Any nontrivial planar graph $G$ of girth at least~5 satisfies $\tau(G) \leq (|V(G)|-2)/3$ by~\cite{KL17} and $\chi_{DP}(G)\le 3$ by~\cite{DP15}, so $\alpha_2^{DP}(G) \geq (2|V(G)|+2)/3 \ge 2|V(G)|/\chi_{DP}(G)$ and $G$ is partially DP-nice.

We will now focus on the proof of Theorem~\ref{thm: subcubic}, which we restate.

\begin{customthm} {\bf \ref{thm: subcubic}}
	Suppose $G$ is a connected, subcubic, triangle-free graph, and $G \neq Q_3$.  Then, $G$ is partially DP-nice.
\end{customthm}

A graph is \emph{subcubic} if its maximum degree is at most~3.  So, $\chi_{DP}(G)\le 3$ for any connected, subcubic graph $G \neq K_4$.  Also, when $G$ is any connected, subcubic graph with $G \neq K_4$, $\tau(G) \leq (3|V(G)|+2)/8$ by~\cite{BH87}, which means $\alpha_2^{DP}(G) \geq (5|V(G)|-2)/8$.  Below we will present a 10-vertex connected subcubic graph $M$ with $\alpha_2^{DP}(M) = 6$, so this bound is sharp.

If $G$ is connected, subcubic, and triangle-free, then $\tau(G) \leq (|V(G)|+1)/3$ by~\cite{BH87}, and $\tau(G) \leq |V(G)|/3$ by~\cite{ZL90} unless $G$ is $V_8$ or $Q_3$.  It follows that $\alpha_2^{DP}(G) \geq 2|V(G)|/3$ and $G$ is partially DP-nice whenever $G$ is a connected, subcubic, triangle-free graph, with the possible exceptions of $V_8$ and $Q_3$.  So, to complete the proof of Theorem~\ref{thm: subcubic} we need only show that $Q_3$ is not partially DP-nice and $V_8$ is partially DP-nice.

\medskip
To get further, we will need an alternative characterization of 2-fold coverings. Consider any graph $G$ with a 2-fold cover $\mathcal{H} = (L,H)$ such that $E_H(L(u),L(v))$ is a perfect matching for all $uv\in E(G)$.  Without loss of generality, suppose that $L(v)=\{v^1,v^2\}$ for all $v\in V(G)$.~\footnote{From this point forward, whenever $\mathcal{H} = (L,H)$ is a 2-fold cover, we assume the vertices of $H$ are named in this manner.}  Then, for each $uv\in E(G)$, either $E_H(L(u),L(v))$ equals $\{u^1v^1,u^2v^2\}$ or $\{u^1v^2,u^2v^1\}$.

We define a \emph{twist representation of $\mathcal H$} to be a function $f:E(G)\rightarrow\{0,1\}$ such that $f(uv)=0$ if $E_H(L(u),L(v))=\{u^1v^1,u^2v^2\}$ and $f(uv)=1$ if $E_H(L(u),L(v))=\{u^1v^2,u^1v^2\}$.  We think of the second case as a ``twist''. This is not unique for $\mathcal H$ because the naming of elements in $L(v)$ is arbitrary; if $v^1$ and $v^2$ were switched then the value of $f$ would flip for all edges incident to $v$.  Nevertheless, we obtain the following characterization.

\begin{lem} \label{lem: sum}
Suppose that $G$ is a graph with a 2-fold cover $\mathcal{H} = (L,H)$ such that $E_H(L(u), L(v))$ is a perfect matching for each $uv \in E(G)$, and $f$ is a twist representation of $\mathcal{H}$. Then $G$ has an $\mathcal{H}$-coloring if and only if for every cycle $C$ in $G$, \[ \sum_{e\in E(C)} f(e) \equiv |E(C)|\ (\bmod\,2). \]
\end{lem}

\begin{proof}
Suppose that $G$ has an $\mathcal{H}$-coloring (i.e., an independent set $S$ in $H$ such that for each $v\in V(G)$, exactly one of $v^1,v^2$ is in $S$).  Let $s(v)$ represent the index of $v$ (i.e., for all $v\in V(G)$, $s(v)=1$ if $v^1\in S$ and $s(v)=2$ if $v^2\in S$).
Then for each edge $uv\in E(G)$, $s(u)\not=s(v)$ if and only if $f(uv)=0$.  
Since traversing a cycle will end where it begins, the value of $s$ must flip between~1 and~2 an even number of times;
hence, any cycle $C$ in $G$ must have an even number of edges $e$ with $f(e)=0$.  Hence, $|E(C)|$ has the same parity as the number of edges $e$ with $f(e)=1$, which equals $\sum_{e\in E(C)} f(e)$, as required.

Now let us assume that $\sum_{e\in E(C)} f(e) \equiv |E(C)|\ (\bmod\,2)$ for every cycle $C$ in $G$.  Let $G'$ be obtained from $G$ by contracting every edge $e$ with $f(e)=1$; note that $G'$ may have multiple edges and loops.  Every cycle $C'$ in $G'$ has a corresponding cycle $C$ in $G$ such that $E(C')\subseteq E(C)$.  Each edge $e\in E(C) - E(C')$ has $f(e)=1$, so $\sum_{e\in E(C)- E(C')} f(e) = |E(C) - E(C')|$.  Then $\sum_{e\in E(C')} f(e) \equiv |E(C')|\ (\bmod\,2)$. Since $\sum_{e\in E(C')} f(e)=0$, $C'$ is an even cycle.  Hence, $G'$ is bipartite.

For each $v\in V(G')$, let $s(v)=1$ or $s(v)=2$ according to its partite set.  We will uncontract edges to recover $G$. Each time we uncontract a vertex $u$ and obtain an edge $vw$, give $s(v)$ and $s(w)$ the same value as $s(u)$.  At the end, $s(v)$ is assigned a value for all $v\in V(G)$ such that $s(u)\not=s(v)$ whenever $f(uv)=0$ and $s(u)=s(v)$ whenever $f(uv)=1$.  Then $\{v^{s(v)}:v\in V(G)\}$ is an ${\mathcal H}$-coloring of $G$, as required.
\end{proof}

We often apply Lemma~\ref{lem: sum} to consider not just $\mathcal H$-colorings but partial $\mathcal H$-colorings, as follows.
Start with the hypotheses of Lemma~\ref{lem: sum} -- that $G$ is a graph with a 2-fold cover $\mathcal{H} = (L,H)$
such that $E_H(L(u), L(v))$ is a perfect matching for each $uv \in E(G)$, and $f$ is a twist representation of $\mathcal{H}$.
Next, let $G'$ be \emph{any induced subgraph} of $G$, let $H' = H[\bigcup_{v \in V(G')} L(v)]$, let $L'$ be $L$ restricted to $V(G')$, and let ${\mathcal H}'=(L', H')$, which is a 2-fold covering of $G'$.  Then restricting $f$ to $E(G')$ gives a twist representation of $\mathcal{H}'$, so $G'$ has an $\mathcal{H}'$-coloring if and only if every cycle $C$ in $G'$ satisfies
$\sum_{e\in E(C)} f(e) \equiv |E(C)|\ (\bmod\,2)$.

We are now ready to complete the proof of Theorem~\ref{thm: subcubic} by proving two lemmas that will address $Q_3$ and $V_8$ respectively.  The \emph{cube graph}~\footnote{Recall that copies of $Q_3$ are isomorphic to the graph with vertex set consisting of all bit strings of length three and edges between two such vertices if and only if the bit strings differ by a single bit.} $Q_3$ contains a cycle $C$ so $\chi_{DP}(Q_3)\ge \chi_{DP}(C)=3$ and it is subcubic so $\chi_{DP}(Q_3)\le \Delta(Q_3)=3$.  Clearly $\tau(Q_3)=3$, so $\alpha_2^{DP}(Q_3)\ge 5$.  We will use a twist representation to show that $\alpha_2^{DP}(Q_3)\le 5$.  It follows that $Q_3$ is not partially DP-nice since $\alpha_2^{DP}=5 < 2|V(Q_3)|/\chi_{DP}(Q_3)=16/3$, and the answer to Question~\ref{ques: PDPCC} is no.

\begin{lem} \label{pro: Q3}
The cube graph $Q_3$ is not partially DP-nice and $\alpha_2^{DP}(Q_3) = 5$.
\end{lem}

\begin{proof}
By the above discussion, it remains to show that $\alpha_2^{DP}(Q_3)\le 5$. Suppose we construct a copy $G$ of $Q_3$ from the following 4-cycles (vertices are written in cyclic order): $x,y,z,w,$ and $x',y',z',w'$, and we add edges that join corresponding vertices (i.e., $x$ to $x'$, $y$ to $y'$, etc.). Define $f:E(G)\rightarrow\{0,1\}$ by letting $f(xy)=f(yz)=f(zw)=f(w'x')=1$ and letting $f$ be 0 on the other 8 edges.  Suppose that $\mathcal{H}=(L,H)$ is a 2-fold cover of $G$ such that $f$ is a twist representation of $\mathcal{H}$.

Note that $G$ has six 4-cycles, each of which has an odd number of edges with $f(e)=1$. An induced subgraph on 6 vertices can omit all those cycles only by omitting a pair of ``opposite corners'': $\{x,z'\}$, $\{y,w'\}$, $\{x',z\}$, or $\{y',w\}$.  Removing such a pair yields an induced 6-cycle, and since each pair is incident to one or three edges $e$ with $f(e)=1$, removing that pair leaves a 6-cycle with three edges or one edge $e$ with $f(e)=1$.  We have shown that any induced subgraph $G'$ of $G$ on 6 vertices will contain a cycle $C$ with $\sum_{e\in E(C)} f(e) \not\equiv |E(C)|\ (\bmod\,2)$.

By applying Lemma~\ref{lem: sum} to $G'$ and the corresponding restriction of $\mathcal{H}$, we conclude that $H$ has no independent set of size 6 (i.e., $\alpha_2^{DP}(G) <6$).
\end{proof}

Finally, we answer Question~\ref{ques: induced} in the negative using the Wagner graph $V_8$, also known as the M\H{o}bius ladder graph on 8 vertices.  We may let $V(V_8)=\{v_1,\ldots,v_8\}$ and let $E(V_8)=\{v_1v_2,v_2v_3,\ldots,v_7v_8,v_8v_1\}\cup\{v_1v_5,v_2v_6,v_3v_7,v_4v_8\}$.  Note that $V_8$, like $Q_3$, is subcubic and contains a cycle so $\chi_{DP}(V_8)=3$. It is easy to check that $\tau(V_8) = 3$ (see~\cite{BH87}).

\begin{lem} \label{pro: counterV8}
$V_8$ is partially DP-nice and $\alpha_2^{DP}(V_8) \geq 6 > |V(V_8)|-\tau(V_8)$.
\end{lem}

\begin{proof}
Let $G_1=V_8 - \{v_3,v_8\}$, $G_2=V_8 - \{v_6,v_8\}$, and $G_3=V_8 - \{v_1,v_7\}$.  Note that each of these contains exactly one cycle, respectively, $C_1$ with vertices (in cyclic order) of $v_1,v_2,v_6,v_5$, $C_2$ with vertices (in cyclic order) of $v_1,v_2,v_3,v_4,v_5$, and $C_3$ with vertices (in cyclic order) of $v_2,v_3,v_4,v_5,v_6$.  Note that every edge in $C_1 \cup C_2 \cup C_3$ appears in exactly two of these cycles.

Suppose that $f$ is a twist representation of an arbitrary 2-fold cover $\mathcal{H}=(L,H)$ of $V_8$.
Let $\mathcal{H}_1=(L_1,H_1)$, $\mathcal{H}_2=(L_2,H_2)$, and $\mathcal{H}_3=(L_3, H_3)$ be appropriate restrictions of $\mathcal{H}=(L,H)$ that are 2-fold covers of $G_1,G_2,G_3$, respectively.  Since every edge in $C_1 \cup C_2 \cup C_3$ appears in exactly two of these cycles,
$$\sum_{e \in E(C_1)} f(e) + \sum_{e \in E(C_2)} f(e) + \sum_{e \in E(C_3)} f(e)$$ is even.
Therefore, it cannot be that $\sum_{e \in E(C_1)} f(e)$ is odd and both $\sum_{e \in E(C_2)} f(e)$ and $\sum_{e \in E(C_3)} f(e)$ are even.  Hence, $\sum_{e\in E(C)} f(e) \equiv |E(C)|\ (\bmod\,2)$ for some $C\in \{C_1,C_2,C_3\}$.  Since each $G_i$ contains only one cycle, every cycle in some $G_i$ satisfies the equivalence, so that $G_i$ has an $\mathcal{H}_i$-coloring.  Thus, $\mathcal{H}$ has an independent set of size $|V(G_i)|=6$.  So, $\alpha_2(V_8)\ge 6$.

Then $\alpha_2(V_8)\ge 6> 16/3 = 2|V(V_8)|/\chi_{DP}(V_8)$, which means $V_8$ is partially DP-nice.
\end{proof}

This completes the proof Theorem~\ref{thm: subcubic}; that is, we can now conclude that every connected, subcubic, triangle-free graph, with the unique exception of $Q_3$, is partially DP-nice.\\

We will finish this section with a construction of  an infinite class of examples for which the answer to Question~\ref{ques: PDPCC} is no.

Let $G$ be the complete graph on 4 vertices $u,v,x,y$ with one edge $xy$ subdivided with degree~2 vertex $z$.
Note that $\tau(G)=2$, which immediately yields $\alpha^{DP}_2(G)\ge 3$.
Let $f(yz)=1$ and let $f$ be 0 on all other edges of $G$.  Suppose $\mathcal{H} = (L,H)$ is a 2-fold cover of $G$ such that $f$ is a twist representation of $\mathcal{H}$.
Note that each 3-cycle in $G$ has even sum of $f(e)$ over its edges.  Furthermore, the 4-cycles in $G$ with vertices in cyclic order of $v,x,z,y$ and $u,x,z,y$ each has odd sum of $f(e)$ over its edges.  Any induced subgraph $G'$ of $G$ on 4 vertices will contain one of those cycles.  So applying Lemma~\ref{lem: sum} to $G'$ and the corresponding restriction of $\mathcal{H}$, we conclude that $H$ has no independent set of size 4 (i.e., $\alpha_2^{DP}(G) <4$).  So, $\alpha^{DP}_2(G)=3$.

Now pick any $n\ge2$ and let $G_i$ be a copy of that same graph $G$ for $i \in [n]$ such that $V(G_i) = \{u_i, v_i, x_i, y_i, z_i \}$ and the function $f:V(G) \rightarrow V(G_i)$ given by $f(t)=t_i$ for each $t \in V(G)$ is a graph isomorphism.  Then for all $1\le i\le n-1$, add the edge $z_iu_{i+1}$, and let $G^*$ be the resulting graph.  Clearly $\tau(G^*)=2n$, so $\alpha^{DP}_2(G^*)\ge 3n$.  By letting $f(y_iz_i)=1$ for all $i\in[n]$ and letting $f$ be 0 on other edges of $G^*$, then as above we have $\sum_{e\in E(C)} f(e) \not\equiv |E(C)|\ (\bmod\,2)$ for each appropriate cycle in each $G_i$.  Since $\tau(G^*)=2n$ and $|V(G^*)|=5n$, any induced subgraph $G'$ with $3n+1$ vertices will contain one of these cycles, so applying Lemma~\ref{lem: sum} to $G'$ allows us to conclude that $\alpha_2^{DP}(G^*) < 3n+1$.  Thus, $\alpha_2^{DP}(G^*) =3n$.

Note that $G^*$ has coloring number~3, which can be seen by considering vertices in the order $u_1,v_1,x_1,y_1,z_1,u_2,\ldots,z_n$; hence $\chi_{DP}(G^*)\le 3$.
Since cycles have DP-chromatic number~3 and $G^*$ contains a cycle, $\chi_{DP}(G^*)\ge3$.  So, $\chi_{DP}(G^*) =3$.
It follows that $G^*$ is not partially DP-nice, since $\alpha_2^{DP}(G^*)=3n < 2|V(G^*)|/\chi_{DP}(G^*)=10n/3$.

\begin{cor} \label{cor: K4+}
For each $n\ge 1$, there is a graph $G^*$ on $5n$ vertices that is not partially DP-nice because $\alpha_2^{DP}(G^*) = 3n < 2|V(G^*)|/3= 10n/3$.
\end{cor}

Finally, notice that if we take $G_1$ and $G_2$ and add an edge between $z_1$ and $z_2$, we obtain a 10-vertex connected subcubic (in fact 3-regular) graph $M$ with $\alpha_2^{DP}(M) = 6$.  This demonstrates the sharpness of the bound $\alpha_2^{DP}(G) \geq (5|V(G)|-2)/8$ for any connected subcubic graph $G$ other than $K_4$, which was mentioned above.

\section{Subadditivity} \label{subadd}

Recall that Question~\ref{ques: PDPCC} asks whether the inequality
\begin{equation}
\alpha_t^{DP}(G) \geq \frac{t|V(G)|}{\chi_{DP}(G)} \tag{$*$}
\end{equation}
holds for all graphs $G$ and all $t \in [\chi_{DP}(G)]$.  
Although we have shown the answer is negative in general, 
in this section we will see how the ideas in~\cite{I10} extend to the context of DP-coloring, which yield results along the lines of Question~\ref{ques: PDPCC}.

The main tool is the following subadditivity lemma.

\begin{customlem} {\bf \ref{lem: triangle}}
For any graph $G$ and $t_1, \ldots, t_k \in \N$, $$\alpha_{t}^{DP}(G) \leq \sum_{i=1}^k \alpha_{t_i}^{DP}(G),$$
where $t=\sum_{i=1}^k t_i$.
\end{customlem}

\begin{proof}
For each $i \in [k]$, let $\mathcal{H}_i = (L_i, H_i)$ be a $t_i$-fold cover of $G$ for which $\alpha(H_i) = \alpha_{t_i}^{DP}(G)$ such that $H_1, \ldots, H_k$ are pairwise vertex disjoint.  For each $v \in V(G)$, let $L(v) = \bigcup_{i=1}^k L_i(v)$.  Let $H$ be the union of $H_1,\ldots,H_k$ with edges added so that each $L(v)$ is a clique.  Then $\mathcal{H} = (L,H)$ is a $t$-fold cover of $G$.

There is an independent set $S$ in $H$ of size $\alpha_t^{DP}(G)$.  For each $i \in [k]$, let $S_i = S \cap V(H_i)$.  Then each $S_i$ is an independent set in $H_i$,  so we have
$$ \alpha_t^{DP}(G) = |S| = \sum_{i=1}^k |S_i| \leq \sum_{i=1}^k \alpha(H_i) = \sum_{i=1}^k \alpha_{t_i}^{DP}(G).$$
\end{proof}

Lemma~\ref{lem: triangle} quickly yields all our tools and bounds.  The first tool states that verifying~$(*)$ for any one $t$ suffices to imply~$(*)$ for all divisors of $t$ as well.

\begin{cor} \label{cor: divides}
Let $G$ be a graph and $s,t \in \N$ such that $t$ divides $s$.
\[
\text{If }\,\alpha_s^{DP}(G) \geq \frac{s|V(G)|}{\chi_{DP}(G)}, \,\text{ then } \,
\alpha_t^{DP}(G) \geq \frac{t|V(G)|}{\chi_{DP}(G)}.
\]
\end{cor}

\begin{proof}
Let $k$ be the integer such that $s=kt$.  By Lemma~\ref{lem: triangle}, $\alpha_s^{DP}(G) \leq k \alpha_t^{DP}(G).$  Then
$\alpha_t^{DP}(G) \geq (1/k)s|V(G)|/\chi_{DP}(G) = t|V(G)|/\chi_{DP}(G)$ as desired.
\end{proof}

Note that applying Corollary~\ref{cor: divides} with $t=2$ and $s=\chi_{DP}(G)$  shows that when $\chi_{DP}(G)$ is even,
inequality~($*$) holds true.  Hence, Question~\ref{ques: twofold} can be restricted to graphs $G$ where $\chi_{DP}(G)$ is odd; this also follows from Theorem~\ref{cor: ceiling}.  Theorem~\ref{cor: ceiling} is similar to what was asked for by Question~\ref{ques: twofold}, and in many cases the bound is close to~$(*)$.

\begin{customthm} {\bf \ref{cor: ceiling}}
For any graph $G$ and $t \in [\chi_{DP}(G)]$,
$$\alpha_t^{DP}(G) \geq \frac{|V(G)|}{\lceil \chi_{DP}(G)/t \rceil}.$$
\end{customthm}

\begin{proof}
Let $s = t \lceil \chi_{DP}(G)/t \rceil$.
Since $s\ge \chi_{DP}(G)$, $\alpha_s^{DP}(G)=|V(G)|$.
By Lemma~\ref{lem: triangle}, $\alpha_s^{DP}(G) \leq{\lceil \chi_{DP}(G)/t \rceil}\alpha_t^{DP}(G)$.
The desired result follows.
\end{proof}

The next result implies Theorem~\ref{cor: half the values}, which states that at least half the elements $t\in [\chi_{DP}(G)-1]$ satisfy inequality~($*$).

\begin{cor} \label{cor: half2}
For a graph $G$ and an integer $t$ with $1 \le t < s = \chi_{DP}(G)$, either
$$\alpha_t^{DP}(G) \geq \frac{t}{s}|V(G)| \quad \text{or} \quad \alpha_{s-t}^{DP}(G) \geq \frac{s-t}{s}|V(G)|.$$
\end{cor}

\begin{proof}
For a contradiction, suppose that $\alpha_t^{DP}(G) < \frac{t}{s}|V(G)|$ and $\alpha_{s-t}^{DP}(G) <\frac{s-t}{s}|V(G)|.$  Then by Lemma~\ref{lem: triangle}, we obtain the contradiction:
$$|V(G)| = \alpha_{s}^{DP}(G) \leq \alpha_{t}^{DP}(G) + \alpha_{s-t}^{DP}(G) < \left({\frac{t}{s} + \frac{s-t}{s}}\right)|V(G)| = |V(G)|.$$\end{proof}

Applying Corollary~\ref{cor: half2} to each $t$ with $1\le t\le \lceil{(\chi_{DP}(G)-1)/2}\rceil$ will yield distinct $t$-values (either $t$ or $\chi_{DP}(G)-t$) that satisfies~($*$), for a total of $\lceil{(\chi_{DP}(G)-1)/2}\rceil$ such $t$-values.
Thus, Theorem~\ref{cor: half the values} is now proven.

\section{Partially DP-nice Graphs} \label{nice}

Recall that a graph $G$ is \emph{partially DP-nice} if~($*$) holds true for all $t \in [\chi_{DP}(G)]$.  Each partially DP-nice graph represents partial progress toward Question~\ref{ques: PDPCC}, which asked whether all graphs might be partially DP-nice.  While many are not, in this section we show that both chordal graphs and series-parallel graphs are, and we also prove Theorem~\ref{thm: join}:
for any graph $G$, there exists a $p \in \N$ such that $G \vee K_p$ is partially DP-nice.

\subsection{Chordal Graphs and Series Parallel Graphs}

Recall that a graph family $\mathcal{G}$ is a \emph{hereditary graph family} if it is closed under taking induced subgraphs.

\begin{pro} \label{pro: hereditary}
Suppose that $\mathcal{G}$ is a hereditary graph family such that for each $G \in \mathcal{G}$, $\chi(G) = \chi_{DP}(G)$.  Then every $G \in \mathcal{G}$ is partially DP-nice.
\end{pro}

\begin{proof}
Let $\mathcal{G}$ be such a graph family. Suppose that $G \in \mathcal{G}$ and let $k=\chi(G) = \chi_{DP}(G)$.
Consider any $t\in [k]$ and any $t$-fold cover $\mathcal{H}=(L,H)$ of $G$.

Consider a proper $k$-coloring of $G$.  Let $S$ be the union of $t$ of the largest color classes.  Then $S$ is a subset of $V(G)$ with $|S|\ge (t/k)|V(G)|$.  Let $\mathcal{H}'=(L', H')$ be the corresponding $t$-fold cover of $G[S]$, i.e., let $L'$ be $L$ restricted to $S$ and let $H'=H[\bigcup_{v \in S} L(v)]$.  Note that $\chi(G[S]) = \chi_{DP}(G[S])$ since $G\in\mathcal{G}$.  So, $H'$ contains an independent set $S'$ of size $|S|$.  Since $H'$ is an induced subgraph of $H$, $S'$ is also an independent set in $H$.  Hence, $\alpha^{DP}_t(G) \ge |S'| \ge (t/k)|V(G)|$ and since $t$ was arbitrarily chosen from $[\chi_{DP}(G)]$, we have that $G$ is partially DP-nice.
\end{proof}

Chordal graphs are such a hereditary graph family, as $\chi(G)=\omega(G)\le \chi_{DP}(G)\le \chi(G)$ for any chordal graph (see~\cite{KM19}). Recall that a graph $G$ is \emph{chordal} if every cycle $C$ in $G$ has a \emph{chord}, which is an edge with endpoints on nonconsecutive vertices of $C$.  Thus, we get the following result.

\begin{cor} \label{cor: chordal}
Chordal graphs are partially DP-nice.
\end{cor}

Given any graph $G$, its treewidth $tw(G)$ can be defined in terms of a chordal graph $M$ formed by adding edges to $G$ so that $M$ has smallest possible clique number; then $tw(G)=\omega(M)-1$.  For example, $tw(G)\le 1$ if and only if a graph is a forest.
Note that $\chi_{DP}(G) \le \chi_{DP}(M) = \omega(M) = tw(G)+1$.

\begin{pro} \label{pro: treewidth}
If $G$ is a graph with $\chi_{DP}(G)=tw(G)+1$, then $G$ is partially DP-nice.
\end{pro}

\begin{proof}
Let $M$ be a chordal graph obtained by adding edges to $G$ such that $\omega(M) = tw(G)+1$.
Since $\chi_{DP}(G)\le \chi_{DP}(M) = \omega(M)$ and $\chi_{DP}(G)=tw(G)+1$, all these are equal.
Let $t\in[\chi_{DP}(G)]$ and let $\mathcal{H} = (L,H)$ be an arbitrary $t$-fold cover of $G$.  Note that $\mathcal{H}$ is also a $t$-fold cover of $M$.
By Corollary~\ref{cor: chordal}, $H$ has an independent set of size at least $t|V(M)|/\chi_{DP}(M)$, which equals
$t|V(G)|/\chi_{DP}(G)$ as required.
\end{proof}

\emph{Series-parallel graphs} are the graphs with treewidth at most 2.  A series parallel graph $G$ that contains a cycle $C$ has $\chi_{DP}(G)\ge \chi_{DP}(C)=3$, so Proposition~\ref{pro: treewidth} applies, showing that $G$ is partially DP-nice.  Any acyclic graph $G$ has coloring number at most~2, so $\chi_{DP}(G)\le 2$, which we have noted means that it must be partially DP-nice.

\begin{cor}
Series-parallel graphs are partially DP-nice.
\end{cor}

\subsection{Join of a Graph with a Complete Graph}\label{sec: join}

Interestingly, partial DP-coloring gets easier when we join a vertex to a graph and the DP-chromatic number of the resulting graph is higher than the DP-chromatic number of the original graph.
Proposition~\ref{pro: easier} illustrates this idea.  First, we need a basic result.

\begin{lem} \label{lem: one vertex}
If $v$ is a vertex in a graph $G$, then $\chi_{DP}(G)-1 \le \chi_{DP}(G-v) \le  \chi_{DP}(G)$.
\end{lem}

\begin{proof}
Clearly, $\chi_{DP}(G-v)\le  \chi_{DP}(G)$.  So, we must show that $\chi_{DP}(G-v) \ge \chi_{DP}(G)-1$.

Let $t=1+ \chi_{DP}(G-v)$ and let $\mathcal{H}=(L,H)$ be any $t$-fold cover of $G$.
Pick a vertex $v'\in L(v)$.  Note that $v'$ has at most one neighbor in each $L(u)$ with $u\not=v$.  Construct $H'$ from $H$ by removing: all of $L(v)$, the neighbors of $v'$ in lists $L(u)$ with $u\not=v$, and one vertex from each list $L(u)$ that has had nothing removed from it yet.  Let $L'$ be $L$ with the same vertices removed and domain $V(G)-v$.  Let $\mathcal{H}'=(L',H')$.  Then $\mathcal{H}'$ is a $(t-1)$-fold cover of $G-v$.
Since $t-1= \chi_{DP}(G-v)$, there is an $\mathcal{H}'$-coloring of $G'$, which is an independent set $S'$ in $H'$ of size $|V(G-v)|=|V(G)|-1$.
Then $S'\cup\{v'\}$ is an independent set in $H$ and it is an $\mathcal{H}$-coloring of $G$.  It follows that $\chi_{DP}(G)\le t$ as required.
\end{proof}

\begin{pro} \label{pro: easier}
Suppose that a graph $G$ is partially DP-nice.  Let $G' = G \vee K_1$.  If $\chi_{DP}(G') > \chi_{DP}(G)$, then $G'$ is partially DP-nice.
\end{pro}

\begin{proof}
Suppose that $\chi_{DP}(G) = m$.  By Lemma~\ref{lem: one vertex}, it must be that $\chi_{DP}(G') = m+1$.  Suppose that $t \in \N$ satisfies $2 \leq t \leq m$.  Also, suppose that $\mathcal{H} = (L,H)$ is a $t$-fold cover of $G'$.  Let $H' = H[ \bigcup_{v \in V(G)} L(v)]$.  Since $G$ is partially DP-nice, we know there is an independent set $I$ in $H'$ of size at least $tn/m$, where $n=|V(G)|$.  Since $n \geq m$, we know that
$$ \frac{tn}{m} \geq \frac{t(n+1)}{m+1}.$$
Consequently, $I$ is an independent set in $H$ of size at least $t(n+1)/(m+1)$.  The desired result follows.
\end{proof}

We know from~\cite{BK18} that for any graph $G$, there exists $\mu \in \N$ such that $\chi_{DP}(G\vee K_\mu) = \chi(G\vee K_\mu) = \chi(G) +\mu$, which is what led us to study whether for such graphs, partial $DP$-niceness might behave like its ordinary coloring analogue.

\begin{pro} \label{pro: one vertex plus}\footnote{This also follows from Theorem~2.1 of~\cite{BK18}, but we include an argument, including Lemma~\ref{lem: one vertex} and its proof, for completeness.}
For any graph $G$ and any $p\ge \mu$, $\chi_{DP}(G \vee K_p) = \chi(G) +p$.
\end{pro}

\begin{proof}
We know that $\chi_{DP}(G  \vee K_\mu) = \chi(G) +\mu$.
For a proof by induction, suppose that $\chi_{DP}(G \vee K_p) = \chi(G) +p$ for some $p\ge \mu$.  Then
$$\chi_{DP}(G\vee K_{p+1})\ge \chi(G\vee K_{p+1}) = \chi(G)+p+1 = \chi_{DP}(G\vee K_p) + 1.$$
We get $\chi_{DP}(G\vee K_{p+1})-1\le \chi_{DP}(G\vee K_p)$ by Lemma~\ref{lem: one vertex}. Therefore, $\chi_{DP}(G\vee K_{p+1}) = \chi(G) +p + 1$, as required.
\end{proof}

The rest of this paper will be devoted to proving Theorem~\ref{thm: join}.  If $G$ is a complete graph, then $G\vee K_p$ is itself a complete graph $K_q$.  Since $\alpha^{DP}_t(K_q)=t = t|V(K_q)|/\chi_{DP}(K_q)$, $K_q$ is partially DP-nice. Thus, we may assume that $G$ is an $n$-vertex graph that is not complete.  Note that then $\chi(G) < n$.

Let $G_0 = G$ and for each $p\ge 1$, let $G_p=G_{p-1}\vee K_1$.  Then $G_{p-1}$ is a subgraph of $G_p$, and $G_p$ is a copy of $G\vee K_p$ for all $p\ge 1$.
We may fix $\mu \in \N$ as in~\cite{BK18} and Proposition~\ref{pro: one vertex plus}.  Then for all $p\ge\mu$, we have $|V(G_p)|/\chi_{DP}(G_p) = (n+p)/(k+p)$, where we have let $k=\chi(G)$.

For each $p \geq \mu$, let $B_p$ be the set of $t\in[k+p]$ for which $\alpha_t^{DP}(G_p) < t(n+p)/(k+p)$.
In other words, $t\in B_p$ if and only if there is a $t$-fold cover $\mathcal{H} = (L,H)$ of $G_p$ such that $H$ has no independent set $S$ with $|S|\ge t(n+p)/(k+p)$.  We know that $B_p$ does not contain $1$ or $k+p$.
Note that $G_p$ is partially DP-nice if and only if $B_p = \emptyset$.

\begin{pro} \label{pro: subset}
If $p \geq \mu$, $B_{p+1} \subseteq B_p$.
\end{pro}

\begin{proof}
Suppose for the sake of contradiction that there exists $t \in B_{p+1} - B_p$.  We know that $1 < t < k+p+1$.  There must be a $t$-fold cover $\mathcal{H} = (L,H)$ of $G_{p+1}$ such that $H$ has no independent set $S$ with $|S|\ge t(n+p+1)/(k+p+1)$. Let $H' = H[ \bigcup_{v \in V(G_p) } L(v)]$, let $L'$ be $L$ restricted to $V(G_p)$, and let $\mathcal{H}'=(L',H')$; then $\mathcal{H}'$ is a $t$-fold cover of $G_p$.  Since $t \notin B_p$ and $t\le k+p$, there must be an independent set $S$ of size at least $t(n+p)/(k+p)$ in $H'$.  Note that $S$ is an independent set in $H$ as well and
$$\frac{t(n+p)}{k+p} \geq \frac{t(n+p+1)}{k+p+1}$$ since $k < n$, which is a contradiction.
\end{proof}

Now we show that there exists $p$ such that $B_p=\emptyset$, which will complete the proof of the theorem.

\begin{customthm} {\bf \ref{thm: join}}
	For any graph $G$, there exists $p \in \N$ such that $G \vee K_p$ is partially DP-nice.
\end{customthm}

\begin{proof}
If not, then by Proposition~\ref{pro: subset} there exists $t\in B_p$ for all $p\ge \mu$.
Since $\lim_{p \rightarrow \infty} (n+p)/(k+p) = 1$, there exists $p\ge \mu$ such that
$$\frac{t(n + p)}{k+p} < t+1.$$ Since $t\in B_p$, $\alpha^{DP}_t(G_p)<t+1$.  So, there is a $t$-fold cover $\mathcal{H} = (L,H)$ of $G_p$ with no independent set of size $t+1$.  Since $1<t<k+p+1$ and $k < n$, we have $t+1\le k +p+1 \le n+p=|V(G_p)|$.
Since $G_p$ is not a complete graph, $G_p$ has an induced subgraph $G'$ on $t+1$ vertices that is not a complete graph.
The coloring number of $G'$ is at most $t$, so for every $t$-fold cover $\mathcal{H}'$ of $G'$, there is an $\mathcal{H}'$-coloring.  In particular, $H$ has an independent set of size $t+1$, which is a contradiction.
\end{proof}



\begin{thebibliography}{99}
{\small

\bibitem{AB80} M. O. Albertson, D. M. Berman, The chromatic difference sequence of a graph, \emph{J. Combin. Theory B} 29(1) (1980), 1-12.

\bibitem{AGH00} M. O. Albertson, S. Grossman, and R. Haas, Partial List Coloring, \emph{Discrete Mathematics} 214 (2000), 235-240.

\bibitem{B16} A. Bernshteyn, The asymptotic behavior of the correspondence chromatic number, \emph{Discrete Mathematics}, 339 (2016), 2680-2692.

\bibitem{B17} A. Bernshteyn, The Johansson-Molloy Theorem for DP-coloring, \emph{Random Structures \& Algorithms} 54(4) (2019), 653-664.

\bibitem{BK182} A. Bernshteyn and A. Kostochka, Sharp Dirac's theorem for DP-critical graphs, \emph{Journal of Graph Theory} 88 (2018) 521-546.

\bibitem{BK19} A. Bernshteyn and A. Kostochka, DP-colorings of hypergraphs, \emph{European Journal of Combinatorics} 78 (2019): 134-146.

\bibitem{BK17} A. Bernshteyn and A. Kostochka, On differences between DP-coloring and list coloring, \emph{Siberian Advances in Mathematics} 21:2 (2018),  61-71.

\bibitem{BK18} A. Bernshteyn, A. Kostochka, and X. Zhu, DP-colorings of graphs with high chromatic number, \emph{European J. of Comb.} 65 (2017), 122-129.

\bibitem{BK172} A. Bernshteyn, A. Kostochka, and S. Pron, On DP-coloring of graphs and multigraphs, \emph{Siberian Mathematical Journal} 58 (2017), 28-36.

\bibitem{BH87} J. A. Bondy, G. Hopkins, and W. Staton, Lower bounds for induced forests in cubic graphs, \emph{Canad. Math. Bull.} 30 (1987), 193-199.

\bibitem{B76} O. V. Borodin, A proof of B. Gr{\"u}nbaum's conjecture on the acyclic 5-colorability of planar graphs, \emph{Dokl. Akad. Nauk SSSR} 231 (1976), 18-20.

\bibitem{C99} G. C. Chappell, A lower bound for partial list colorings, \emph{Journal of Graph Theory} 32(4) (1999), 390-393.

\bibitem{DP15} Z. Dvo\v{r}\'{a}k and L. Postle, Correspondence coloring and its application to list-coloring planar graphs wihtout cycles of lengths 4 to 8, \emph{J. Combinatorial Theory Series B} 129 (2018), 38-54.

\bibitem{ET79} P. Erd\H{o}s, A. L. Rubin, and H. Taylor, Choosability in graphs, \emph{Congressus Numerantium} 26 (1979), 125-127.

\bibitem{HH03} R. Haas, D. Hanson, and G. MacGillivray, Bounds for partial list colourings, \emph{Ars Combinatoria} 67 (2003), 27-31.

\bibitem{I10} M. N. Iradmusa, A note on partial list coloring, \emph{Australasian Journal of Combinatorics} 46 (2010), 19-24.

\bibitem{J01} J. C. M. Janssen, A partial solution of a partial list colouring problem, \emph{Congressus Numerantium} (2001), 75-80.

\bibitem{JM15} J. Janssen, R. Mathew, and D. Rajendraprasad, Partial list colouring of certain graphs, \emph{The Electronic Journal of Combinatorics} 22(3) (2015), P3.41.

\bibitem{KM19} H. Kaul and J. Mudrock, On the chromatic polynomial and counting DP-colorings of graphs, to appear in \emph{Advances in Applied Mathematics} 123 (2021).

\bibitem{KM20} H. Kaul and J. Mudrock, Combinatorial Nullstellensatz and DP-coloring of Graphs, \emph{Discrete Mathematics} 343 (2020), 112115.


\bibitem{KL17} T. Kelly and C-H. Liu, Minimum size of feedback vertex sets of planar graphs of girth at least five, \emph{European J. of Comb.} 61 (2017), 138-150.

\bibitem{KO18} S-J. Kim and K. Ozeki, A note on a Brooks' type theorem for DP-coloring, \emph{Journal of Graph Theory} 91(2) (2018), 148-161.

\bibitem{KO182} S-J. Kim and K. Ozeki, A sufficient condition for DP-4-colorability, \emph{Discrete Mathematics} 341 (2018), 1983-1986.


\bibitem{LL19} R. Liu and X. Li., Every planar graph without 4-cycles adjacent to two triangles is DP-4-colorable, \emph{Discrete Mathematics} 342 (2019), 623-627.

\bibitem{LLYY19} R. Liu, S. Loeb, Y. Yin, and G. Yu, DP-3-coloring of some planar graphs, \emph{Discrete Mathematics} 342 (2019), 178-189.


\bibitem{M18} J. Mudrock, A note on the DP-chromatic number of complete bipartite graphs, \emph{Discrete Mathematics} 341 (2018) 3148-3151.

\bibitem{V76} V. G. Vizing, Coloring the vertices of a graph in prescribed colors, \emph{Diskret. Analiz.} no. 29, \emph{Metody Diskret. Anal. v Teorii Kodovi Skhem} 101 (1976), 3-10.

\bibitem{V93} M. Voigt, Algorithmic aspects of partial list colorings, \emph{Combinatorics, Probability, and Computing} 11 (1993), 1-10.

\bibitem{W01} D. B. West, (2001) \emph{Introduction to Graph Theory}.  Upper Saddle River, NJ: Prentice Hall.

\bibitem{ZL90} M. L. Zheng and X. Lu, On the maximum induced forestes of a connected cubic graph without triangles, \emph{Discrete Mathematics} 85 (1990), 89-96.
}

\end{thebibliography}
\end{document}